\documentclass[final,12pt]{elsarticle}
\usepackage{srcltx}
\usepackage{eurosym}
\usepackage{mathtools}
\usepackage{amsmath}
\usepackage{amsfonts}
\usepackage{amssymb}
\usepackage{amsthm}
\usepackage{graphicx}
\usepackage{mathrsfs}
\usepackage{xcolor}
\usepackage{exscale}
\usepackage{latexsym}
\usepackage[colorlinks,plainpages=true,pdfpagelabels,hypertexnames=true,colorlinks=true,pdfstartview=FitV,linkcolor=blue,citecolor=red,urlcolor=black]{hyperref}
\PassOptionsToPackage{unicode}{hyperref}
\PassOptionsToPackage{naturalnames}{hyperref}
\usepackage{enumerate}
\usepackage[shortlabels]{enumitem}
\usepackage{bookmark}
\usepackage{wasysym}
\usepackage{esint}
\usepackage[ddmmyyyy]{datetime}
\usepackage[margin=1in]{geometry}
\makeatletter
\g@addto@macro\normalsize{%
  \setlength\abovedisplayskip{4pt}
  \setlength\belowdisplayskip{4pt}
  \setlength\abovedisplayshortskip{4pt}
  \setlength\belowdisplayshortskip{4pt}
}
\everymath{\displaystyle}
\usepackage[capitalize,nameinlink]{cleveref}
\crefname{section}{Section}{Sections}
\crefname{subsection}{Subsection}{Subsections}
\crefname{condition}{Condition}{Conditions}
\crefname{hypothesis}{Hypothesis}{Conditions}
\crefname{assumption}{Assumption}{Assumptions}
\crefname{lemma}{Lemma}{Lemmas}
\crefname{definition}{Definition}{Definitions}

\crefformat{equation}{\textup{#2(#1)#3}}
\crefrangeformat{equation}{\textup{#3(#1)#4--#5(#2)#6}}
\crefmultiformat{equation}{\textup{#2(#1)#3}}{ and \textup{#2(#1)#3}}
{, \textup{#2(#1)#3}}{, and \textup{#2(#1)#3}}
\crefrangemultiformat{equation}{\textup{#3(#1)#4--#5(#2)#6}}%
{ and \textup{#3(#1)#4--#5(#2)#6}}{, \textup{#3(#1)#4--#5(#2)#6}}%
{, and \textup{#3(#1)#4--#5(#2)#6}}

\Crefformat{equation}{#2Equation~\textup{(#1)}#3}
\Crefrangeformat{equation}{Equations~\textup{#3(#1)#4--#5(#2)#6}}
\Crefmultiformat{equation}{Equations~\textup{#2(#1)#3}}{ and \textup{#2(#1)#3}}
{, \textup{#2(#1)#3}}{, and \textup{#2(#1)#3}}
\Crefrangemultiformat{equation}{Equations~\textup{#3(#1)#4--#5(#2)#6}}%
{ and \textup{#3(#1)#4--#5(#2)#6}}{, \textup{#3(#1)#4--#5(#2)#6}}%
{, and \textup{#3(#1)#4--#5(#2)#6}}

\crefdefaultlabelformat{#2\textup{#1}#3}

\newtheorem{theorem} {Theorem}[section]

\newtheorem{lemma}[theorem]{Lemma}

\newtheorem{example}[theorem]{Example}
\newtheorem{counter example}[theorem]{Counter Example}

\newtheorem{definition}[theorem] {Definition}

\def\CC{{\rm \kern.24em \vrule width.02em height1.4ex depth-.05ex \kern-.26emC}}

\def\TagOnRight

\def\AA{{it I} \hskip-3pt{\tt A}}

\def\QQ{\rlap {\raise 0.4ex \hbox{$\scriptscriptstyle |$}} {\hskip -0.1em Q}}

\makeatletter
\newcommand{\vo}{\vec{o}\@ifnextchar{^}{\,}{}}
\makeatother

\def\YYint#1#2#3{{\setbox0=\hbox{$#1{#2#3}{\iint}$}
    \vcenter{\hbox{$#2#3$}}\kern-.50\wd0}}

\def\XXint#1#2#3{{\setbox0=\hbox{$#1{#2#3}{\int}$}
    \vcenter{\hbox{$#2#3$}}\kern-.50\wd0}}

\makeatletter
\def\namedlabel#1#2{\begingroup
   \def\@currentlabel{#2}%
   \label{#1}\endgroup
}
\makeatother
\makeatletter
\newcommand{\rmh}[1]{\mathpalette{\raisem@th{#1}}}
\newcommand{\raisem@th}[3]{\hspace*{-1pt}\raisebox{#1}{$#2#3$}}
\makeatother

\newcounter{desccount}
\newcommand{\descitem}[2]{\item[#1]\refstepcounter{desccount}\label{#2}}
\newcommand{\descref}[2]{\hyperref[#1]{\textnormal{\textcolor{black}{}\textcolor{blue}{\bf #2}\textcolor{black}{}}}}

\newcommand{\dref}[2]{\hyperref[#1]{\textcolor{black}{(}\textcolor{blue}{\bf #2}\textcolor{black}{)}}}
\newcommand{\be} {\begin{eqnarray}}
\newcommand{\ee} {\end{eqnarray}}
\newcommand{\Bea} {\begin{eqnarray*}}
\newcommand{\Eea} {\end{eqnarray*}}
\newcommand{\pa} {\partial}
\newcommand{\re}{\mathbb{R}}

 \newcommand{\al} {\alpha}
\newcommand{\rr}{\rightarrow}

\newcommand{\de} {\delta}

\newcommand{\p}  {\prime}
\newcommand{\e}  {\epsilon}

\newcommand{\la} {\lambda}
\newcommand{\si} {\sigma}

\newcommand{\f}{\infty}
\newcommand{\R}{\mathbb{R}}

\newcommand{\lab} {\label}

\newcommand{\Z}{\mathbb{Z}}

\newcommand{\D}{\Delta}
\newcommand{\jph}{j+1/2}
\newcommand{\jmh}{j-1/2}

\newcommand{\BV}{\textrm{BV}}

\DeclareMathOperator{\loc}{loc}

\DeclareMathOperator{\TV}{TV}

\newcommand{\norm}[1]{\left|\hspace{-0.2mm}\left| #1 \right|\hspace{-0.2mm}\right|}
\newcommand{\abs}[1]{\left| #1\right|}

\newcommand{\sumj}{\sum_{j \in \Z}}
\newcommand{\mF}{\mathcal{F}}
\newcommand{\mM}{\mathcal{M}}
\newcounter{whitney}
\refstepcounter{whitney}

\newcounter{ineqcounter}
\refstepcounter{ineqcounter}
\makeatletter
\def\ps@pprintTitle{%
\let\@oddhead\@empty
\let\@evenhead\@empty
\def\@oddfoot{}%
\let\@evenfoot\@oddfoot}
\makeatother
\usepackage[doublespacing]{setspace}
\usepackage[titletoc,toc,page]{appendix}

\makeatletter
\newcommand{\refcheckize}[1]{%
  \expandafter\let\csname @@\string#1\endcsname#1%
  \expandafter\DeclareRobustCommand\csname relax\string#1\endcsname[1]{%
    \csname @@\string#1\endcsname{##1}\wrtusdrf{##1}}%
  \expandafter\let\expandafter#1\csname relax\string#1\endcsname
}
\makeatother

\refcheckize{\cref}
\refcheckize{\Cref}


\makeatletter
\newcommand{\mainsectionstyle}{%
	\renewcommand{\@secnumfont}{\bfseries}
	\renewcommand\section{\@startsection{section}{2}%
		\z@{.5\linespacing\@plus.7\linespacing}{-.5em}%
		{\normalfont\bfseries}}%
}
\makeatother
\usepackage{pgf,tikz}
\usetikzlibrary{arrows}
\usetikzlibrary{decorations.pathreplacing}

\usepackage{xpatch}
\xpatchcmd{\MaketitleBox}{\hrule}{}{}{}
\xpatchcmd{\MaketitleBox}{\hrule}{}{}{}

\date{}

\begin{document}
\begin{frontmatter}
	
	\title{Convergence of a Godunov scheme to an Audusse-Perthame adapted entropy
	solution for conservation laws with BV spatial flux }
	
	\author[myaddress1]{Shyam Sundar Ghoshal}
	\ead{ghoshal@tifrbng.res.in}

	\author[myaddress1]{Animesh Jana}
	\ead{animesh@tifrbng.res.in }
	
	\address[myaddress1]{Centre for Applicable Mathematics,Tata Institute of Fundamental Research, Post Bag No 6503, Sharadanagar, Bangalore - 560065, India.}

	\author[myaddress2]{John D. Towers}
	\ead{john.towers@cox.net}
	
	\address[myaddress2]
	{MiraCosta College, 3333 Manchester Avenue, Cardiff-by-the-Sea, CA 92007-1516, USA.}
	\begin{abstract}
	In this article we consider the initial value problem for a scalar conservation law in one space dimension with a
	spatially discontinuous flux. There may be infinitely many flux discontinuities, and the set of discontinuities may have
	accumulation points. Thus the existence of traces cannot be assumed. In \cite{AudussePerthame} Audusse and Perthame 	proved a uniqueness result that does not
	require the existence of traces, using adapted entropies.  We generalize the Godunov-type scheme of Adimurthi, Jaffr\'e 	and Gowda \cite{AJG} for this problem with the following assumptions on the flux function, (i) the flux is BV in the spatial 	variable and (ii) the critical point of the flux is BV as a function of the space variable. We prove that the Godunov 	approximations converge to an adapted entropy solution, thus providing an existence result, and extending the convergence 	result of  Adimurthi, Jaffr\'e and Gowda.		
	\end{abstract}
	\begin{keyword}
		conservation law \sep discontinuous flux \sep existence \sep numerical scheme \sep singular mapping 		\sep entropy inequality \sep Godunov scheme. 
    \MSC[2010] 35L65, 35B44, 35A01, 65M06, 65M08
	\end{keyword}
	
\end{frontmatter}
\vspace*{-.5cm}	
\section{Introduction}	 
In this article we prove existence of an adapted entropy solution in the sense of
Audusse and Perthame \cite{AudussePerthame}, via a convergence proof for a Godunov type
finite difference scheme, to the following Cauchy problem:
\begin{eqnarray}
\pa_t u+\pa_x A(u,x)&=&0\hspace*{.55cm}\ \mbox{ for }(x,t)\in \R \times (0,T)=: \Pi_T,\lab{11}\\
u(x,0)&=&u_0\hspace*{.5cm}\mbox{ for }x\in\re.\lab{ini12}
\end{eqnarray}
The partial differential equation appearing above is a scalar one-dimensional conservation law whose flux
$A(u,x)$ has a spatial dependence that may have infinitely many spatial discontinuities. In contrast to all
but a few previous papers on conservation laws with discontinuous flux that address the uniqueness question, 
we make no assumption about the existence
of traces, and so the set of spatial flux discontinuities could have accumulation points.

Scalar conservation laws with discontinuous flux have a number of
applications including vehicle traffic flow with rapid transitions in road conditions \cite{GNPT:2007},
sedimentation in clarifier-thickeners \cite{bkrt2,Diehl:1996B},
oil recovery simulation  \cite{shen_polymer},
two phase flow in porous media \cite{And_Cances},
and granular flow \cite{may_shearer}.

Even in the absence of spatial flux discontinuities, solutions of conservation laws develop
discontinuities (shocks). Thus we seek weak solutions, which are bounded measurable functions
$u$ satisfying \eqref{11} in the sense of distributions. Closely related to the presence of shocks is the problem of nonuniqueness.
Weak solutions are not generally unique without an additional condition or conditions, so-called
entropy conditions. For the classical case of a conservation law with a spatially independent flux
\begin{equation}\label{cl_no_jumps}
u_t + f(u)_x = 0,
\end{equation} 
one requires that the
Kru\v{z}kov entropy inequalities hold in the sense of distributions:
\begin{equation}\label{kruzkov_classical}
\partial_t \abs{u-k} + \partial_x \{sgn(u-k)(f(u)-f(k)) \} \le 0, \quad \forall k \in \R,
\end{equation}
and then uniqueness follows from \eqref{kruzkov_classical}.

There are two main difficulties that arise which are not present
in the classical case \eqref{cl_no_jumps}.  The first problem is existence, the new difficulty being 
that getting a $TV$ bound for the solution with $BV$ initial data may not be possible in general due to the counter-examples given in \cite{ADGG,SSG1}. More interestingly, a $TV$ bound for the solution is possible near the interface for non-uniformly convex fluxes (see Reference \cite{SSG2}).
Several methods have been used to deal with the lack of a spatial variation bound, the main ones
being the so-called singular mapping, compensated compactness, and a local variation bound.
In this paper we employ the singular mapping approach, applied to approximations generated
by a Godunov type difference scheme.
The singular mapping technique is used to get a $TV$ bound of a transformed (via the singular mapping) quantity.  Once the $TV$ bound of the transformed quantity is achieved we can pass to the limit and get a solution satisfying the adapted entropy inequality. Showing that the limit of the numerical approximations satisfies the adapted entropy inequalities is not straightforward due
to the presence of infinitely many flux discontinuities. 

The second problem is uniqueness. The usual Kru\v{z}kov entropy inequalities
do not apply to the discontinuous case. Also, it turns out that there are many reasonable notions
of entropy solution \cite{AMG}, \cite{AKR}. One must consider the application in order to decide on which definition of
entropy solution to use. 

There have been many papers on the subject of scalar conservation laws with
spatially discontinuous flux over
the past several decades.
Most papers on this subject that have addressed the uniqueness question have assumed a 
finite number of flux discontinuities. Often the case of a single flux discontinuity is 
addressed, with the understanding that the results are readily extended to the case of
any finite number of flux discontinuities. The admissibility condition has usually boiled down
to a so-called interface condition (in addition to the Rankine-Hugoniot condition) that involves
the traces of the solution across the spatial flux discontinuity. Often the interface condition
consists of one or more inequalities, and is often derived from some modified version of the classical
Kru\v{z}kov entropy inequality.

When there are only finitely many flux discontinuities, existence of the required traces is
guaranteed, assuming that $u \mapsto A(u,x)$ is genuinely nonlinear \cite{Panov:2006uq}, \cite{Vasseur}.  
However if there are 
infinitely many flux discontinuities, and the subset of $\R$ where they 
occur has one or more accumulation points, these existence results for traces do not apply. Thus a definition of
entropy solution which does not refer to traces is of great interest. 

A method has been developed first in \cite{BaitiJenssen}, and then extended in \cite{AudussePerthame},
using so-called adapted entropy inequalities, that provides a notion of entropy solution and does
not require the existence of traces. For the conservation law $u_t + A(u,x)_x = 0$ with $x \mapsto A(u,x)$ smooth,
the classical Kru\v{z}kov inequality \eqref{kruzkov_classical} becomes
\begin{equation}\label{kruzkov_smooth}
\partial_t \abs{u-k} + \partial_x \{sgn(u-k)(A(u,x)-A(k,x)) \} + sgn(u-k)\partial_x A(k,x)
\le 0, \quad \forall k \in \R.
\end{equation}
Due to the term $sgn(u-k)\partial_x A(k,x)$, this definition does not make 
sense without modification when one tries to
extend it to the case of the discontinuous flux $A(u,x)$ considered here.

The adapted entropy approach consists of replacing the constants $k \in \R$ by
functions $k_\alpha$ defined by the equations
\begin{equation*}
A(k_\alpha(x),x) = \alpha, \quad x\in \R.
\end{equation*}
With this approach the troublesome term $sgn(u-k)\partial_x A(k,x)$ is not present, and
the definition of adapted entropy solution is
\begin{equation}\label{kruzkov_adapted}
\partial_t \abs{u-k_{\alpha}} + 
\partial_x \{sgn(u-k_{\alpha})(A(u,x)-\alpha) \} \le 0.
\end{equation}
Reference \cite{BaitiJenssen} used this approach for the closely related problem where
$u\mapsto A(u,x)$ is strictly increasing. They proved both existence and uniqueness,
with the additional assumption that the flux has the form $A(u,x) = \tilde{A}(u,v(x))$.
Reference~\cite{AudussePerthame} proved uniqueness for both the unimodal case considered
in this paper, along with the case where $u\mapsto A(u,x)$ is strictly increasing.
The existence question was left open.

Recently there has been renewed interest in the existence question for problems where
the Audusse-Perthame uniqueness theory applies. Reference \cite{piccoli_tournus}
proved existence for the problem where
$u\mapsto A(u,x)$ is strictly increasing, and without assuming the
special functional form $A(u,x) = \tilde{A}(u,v(x))$. This was accomplished under the
simplifying assumption that $u \mapsto A(u,x)$ is concave. Reference \cite{Towers_bv}
extended the result of \cite{piccoli_tournus} to the case where 
$u \mapsto A(u,x)$ is not required to be concave. Panov \cite{Panov2009a} proved existence of an adapted entropy solution, 
under assumptions that include our setup, by a measure-valued solution approach. The approach of
\cite{Panov2009a} is quite general but more abstract than ours, and is not associated
with a numerical algorithm.  


 The Godunov type scheme of this paper is a generalization of the scheme developed in \cite{AJG} for the case where
 the flux has the form
 \begin{equation}\label{two_flux}
 A(u,x) = g(u) (1-H(x)) + f(u) H(x),
 \end{equation}
 where each of $g, f$ is unimodal and $H(\cdot)$ denotes the Heaviside function. This is a so-called two-flux problem, where
 there is a single spatial flux discontinuity. The authors of \cite{AJG} proposed a very simple interface flux that extends the classical
 Godunov flux so that it properly handles a single flux discontinuity. The singular mapping technique is used to prove that the Godunov approximations
 converge to a weak solution of the conservation law. With an additional regularity assumption about the limit solution (the solution is assumed   to be continuous except for finitely many Lipschitz curves in $\R \times \R_+$), they also
 prove uniqueness. 
 
 The scheme and results of
 the present paper improve and extend those of \cite{AJG}. By adopting the Audusse-Perthame definition of entropy solution \cite{AudussePerthame}, and then invoking
 the uniqueness result of \cite{AudussePerthame}, we are able to remove the regularity assumption employed in  \cite{AJG}, and also the restriction to finitely many flux discontinuities. Moreover, the scheme of \cite{AJG} is defined on a nonstandard spatial grid that is specific to the case of a single flux discontinuity, and would be inconvenient from a programming viewpoint for the case of multiple flux discontinuities. Our scheme uses a
 standard spatial grid, and in fact our algorithm does not require that flux discontinuities be specifically located, identified, or processed in any special way. Our approach is based on the observation that it is possible to simply apply the Godunov interface flux at every grid cell boundary. At cell boundaries where there is no flux discontinuity, the interface flux automatically reverts to the classical Godunov flux, as desired. This not only makes it possible to use a standard spatial grid, but also simplifies the analysis of the scheme.
 
 The remainder of the paper is organized as follows.
 In Section~\ref{def} we specify the assumptions on the data of the problem, give the definition of adapted entropy solution, 
 and state our main theorem, Theorem~\ref{theorem1}.
 In Section~\ref{section_compactness_direct} we give the details of 
the Godunov numerical scheme, and prove convergence (along a subsequence) of the resulting approximations. In Section~\ref{entropy_section_A} we show that a (subsequential) limit solution guaranteed by our convergence theorem is an adapted entropy solution in the sense of Definition~\ref{entropy_solution},
completing the proof of the main theorem.

\section{Main theorem}\label{def}
We assume that the flux function $A:\re\times\re\rr\re_+$ satisfies the following conditions:
\begin{description}
	\descitem{H-1}{H1} For some $r>0$
	\begin{equation*}
	|A(u_1,x)-A(u_2,x)|\leq C|u_1-u_2|\mbox{ for }u_1,u_2\in[-r,r]
	\end{equation*}
	where the constant $C=C(r)$ is independent of  $x$.
	\descitem{H-2}{H2} There is a $BV$ function $a:\re\rr\re$ and a continuous
	function $R:\R \rightarrow \R^+$ such that
	\begin{equation*}
	|A(u,x)-A(u,y)|\leq R(u)|a(x)-a(y)|.
	\end{equation*}
	\descitem{H-3}{H3} For each $x\in\re$ the function
	$u \mapsto A(u,x)$ is unimodal, meaning that there is $u_{M}(x)\in \re$ such that
	$A(u_M(x),x)=0$ and $A(\cdot,x)$ is decreasing on $(-\f,u_M(x)]$ and increasing on $[u_M(x),\f)$.
	We further assume that there is a continuous function $\gamma: [0,\infty) \rightarrow [0,\infty)$,
	which is strictly increasing with $\gamma(0)=0$, $\gamma(+\infty) = +\infty$, and
	such that
	\begin{equation}\label{uniform_unimodal}
	\begin{split}
	&\textrm{$A(u,x) \ge \gamma(u-u_M(x))$ for all $x\in \R$ and $u \in [u_M(x),\infty]$},\\
	&\textrm{$A(u,x) \ge \gamma(-(u-u_M(x)))$ for all $x\in \R$ and $u \in (-\infty,u_M(x)]$}.
	\end{split}
	\end{equation}
	\descitem{H-4}{H4} $u_M \in BV(\R)$.	
\end{description}	
Above we have used the notation $\BV(\R)$ to denote the set of functions of bounded variation on $\R$, i.e., those functions
$\rho:\R \mapsto \R$ for which 
\begin{equation*}
\TV(\rho) := \sup \left\{\sum_{k=1}^K \abs{\rho(\xi_k) - \rho(\xi_{k-1})} \right\} < \infty,
\end{equation*}
where the $\sup$ extends over all $K\ge 1$ and all partitions $\{\xi_0< \xi_1< \ldots < \xi_K \}$ of $\R$.

By Assumption \descref{H3}{H-3}, for each $\alpha \ge 0$ there exist two functions $k_\al^+(x)\in[u_M(x),\f)$ and $k_\al^-(x)\in(-\f,u_M(x)]$ uniquely determined from the following equations:
\begin{equation}
A(k_\al^+(x),x)= A(k_\al^-(x),x)=\al.
\end{equation}

Related to the flux $A(\cdot,\cdot)$ is 
the so-called singular mapping:
\begin{equation}\label{Psi}
\Psi(u,x):=\int\limits_{u_{M}(x)}^{u}\left|\frac{\pa}{\pa u}A(\theta,x)\right|d\theta.
\end{equation}
It is clear that for each $x\in \R$ the mapping $u \mapsto \Psi(u,x)$ is strictly increasing. Therefore for each $x\in\re$ the map $u\mapsto\Psi(u,x)$ is invertible and we denote the inverse map by $\al(u,x)$. Notice that $\al(\cdot,\cdot)$ and $\Psi(\cdot,\cdot)$ satisfy the following relation
\begin{equation}\label{map_alpha}
\Psi(\al(u,x),x)=u=\al(\Psi(u,x),x)\mbox{ for all }x\in\re.
\end{equation}
Also, due to Assumption \descref{H3}{H-3},  \eqref{Psi} is equivalent to $\Psi(u,x)=sgn(u-u_M(x))A(u,x)$. 
\begin{definition}\lab{entropy_solution}
	A function $u \in L^{\infty}(\Pi_T) \cap C([0,T]:L^1_{\loc}(\R))$ is an adapted entropy solution of  
	the Cauchy problem (\ref{11})--(\ref{ini12})
	if it satisfies the following adapted entropy inequality in the sense of distributions:
	\begin{equation}\label{AEC}
	\pa_t|u-k_\al^\pm(x)|+\pa_x \left[sgn(u-k_\al^\pm(x))(A(u,x)-\al)\right]\leq 0
	\end{equation}  
	for all $\al\geq 0$ or equivalently,

\begin{equation}\label{kruzkov_type}
\begin{split}
\int\limits_{\re_+}\int\limits_{\re}\frac{\pa\phi}{\pa t}|u(x,t)-k^{\pm}_{\al}(x)|\,dxdt
&+\int\limits_{\re_+}\int\limits_{\re}\frac{\pa \phi}{\pa x}sgn(u(x,t)-k^{\pm}_{\al}(x))(A(u(x,t),x)-\al)\,dxdt \\
&+ \int\limits_{\re} \abs{u_0(x) -k^{\pm}_{\al}(x)}\phi(x,0) \,dx 
\geq 0 
\end{split}
\end{equation}
for any $0\leq\phi\in C_c^{\f}(\re\times [0,\infty))$.
\end{definition}
For uniqueness and stability we will rely on the following result by Panov.
\begin{theorem}[Uniqueness Theorem \cite{Panov2009a}]\label{thm:uniqueness}
	Let $u, v$ be adapted entropy solutions in the sense of Definition \ref{entropy_solution}, 
	with corresponding initial data $u_0$, $v_0$,
	and assume that Assumptions (\descref{H1}{H-1})--(\descref{H4}{H-4}) hold. 
	 Then for a.e. $t\in[0,T]$ and any $r>0$ we have
	\begin{equation}\label{contraction}
	\int\limits_{\abs{x}\leq r}\abs{\al(u(x,t),x)-\al(v(x,t),x)}\,dx\leq \int\limits_{\abs{x}\leq r+L_1t}\abs{\al(u_0(x),x)-\al(v_0(x),x)}\,dx
	\end{equation}
	where $L_1:=\sup\{\pa_u A(u,x);\,x\in\re,\abs{u}\leq\max(\|u_0\|_{L^{\f}},\|v_0\|_{L^{\f}})\}$ and $\al$ is as in \eqref{map_alpha}.
\end{theorem}
Though Theorem \ref{thm:uniqueness} is not stated in \cite{Panov2009a} but it essentially follows from the techniques used in \cite[Theorem 2]{Panov2009a} and Kru\v{z}kov's uniqueness proof \cite{Kruzkov} for scalar conservation laws. For sake of completeness we give a sketch of the proof for Theorem \ref{thm:uniqueness} in Appendix. The main reason to rely on Theorem \ref{thm:uniqueness} instead of the uniqueness result in \cite{AudussePerthame} is to exclude the following assumption \cite[Hypothesis (H1); page 5]{AudussePerthame} on flux:
\begin{itemize}
	\item A(u,x) is continuous at all points of $\R \times \R \setminus \mathcal{N}$ where $\mathcal{N}$ is a closed set of
	zero measure.
\end{itemize}

Reference \cite{AudussePerthame} presents the following two examples to which their uniqueness theorem applies.
\begin{example}
	\begin{equation*}
	A(u,x) = S(x)u^2, \quad S(x) >0.
	\end{equation*}
	In this example $u_M(x) = 0$ for all $x \in \R$. Assumptions (\descref{H1}{H-1})--(\descref{H4}{H-4}) 
	are satisfied if $S \in BV(\R)$, and $S(x) \ge \epsilon$ for some $\epsilon>0$.
\end{example}

\begin{example}
	\begin{equation*}
	A(u,x) = (u-u_M(x))^2.
	\end{equation*}
	Assumptions (\descref{H1}{H-1})--(\descref{H4}{H-4}) 
	are satisfied for this example also if we assume that 
	$u_M \in BV(\R).$
\end{example}
Our main theorem is
\begin{theorem}\label{theorem1}
	Assume that the flux function $A$ satisfies (\descref{H1}{H-1})--(\descref{H4}{H-4}), and that $u_0\in L^{\infty}(\R)$. 
	Then as the mesh size $\D \rightarrow 0$, the approximations $u^{\D}$ generated by the Godunov scheme described in 	Section~\ref{section_compactness_direct} converge in $L^1_{\loc}(\Pi_T)$ and pointwise a.e in $\Pi_T$ to the unique 	adapted entropy solution $u \in L^{\infty}(\Pi_T) \cap C([0,T]:L^1_{\loc}(\R))$ corresponding to the Cauchy problem 	(\ref{11})--(\ref{ini12}) with initial data $u_0$.
\end{theorem}


\section{Godunov scheme and compactness}\label{section_compactness_direct}
For $\D x>0$ and $\D t>0$ consider equidistant spatial grid points $x_j:=j\D x$ for $j\in\Z$ and temporal grid points $t^n:=n\D t$ 
for integers $0 \le n\le N$. Here $N$ is the integer such that $T \in [t^N,t^{N+1})$. Let $\la:=\D t/\D x$. We fix the notation $\chi_j(x)$ for the indicator function of $I_j:=[x_j - \D x /2, x_j + \D x /2)$, and
$\chi^n(t)$ for the indicator function of $I^n:=[t^n,t^{n+1})$. Next we approximate initial data $u_0\in L^{\f}$ by a piecewise constant function $u^{\D}_0$ defined as follows:
\begin{equation}
u^{\D}_0:=\sumj\chi_j(x)u^0_j\quad \mbox{where }u^0_j=\frac{1}{\D x}\int_{I_j}u_0(x)\,dx\mbox{ for }j\in\Z.
\end{equation}
The  approximations generated by the scheme are denoted by $u_j^n$, where $u_j^n \approx u(x_j,t^n)$.
The grid function $\{u_j^n\}$ is extended to a function defined on $\Pi_T$ via
\begin{equation}\label{def_u_De}
u^{\D}(x,t) =\sum_{n=0}^N \sumj \chi_j(x) \chi^n(t) u_j^n.
\end{equation}
The Godunov type scheme that we employ is then:
\begin{equation}\label{scheme_A}
u_j^{n+1} = u_j^n - \lambda \D_- \bar{A}(u^n_j,u^n_{j+1},x_j,x_{j+1}), \quad j \in \Z, \quad n=0,1,\ldots,N,
\end{equation}
where the numerical flux $\bar{A}$ is
\begin{equation}\label{def_bar_A_direct}
\bar{A}(u,v,x_j,x_{j+1}) :=
\max \left\{A(\max(u,u_M(x_j)),x_j) , A(\min(v,u_M(x_{j+1})),x_{j+1})  \right\}.
\end{equation}
When $A(\cdot,x_j)=A(\cdot,x_{j+1})$, the flux $\bar{A}$ reduces to the classical Godunov flux that is used for conservation laws
where the flux does not have a spatial dependence. Otherwise $\bar{A}$ is a generalization of the Godunov flux
proposed in \cite{AJG} for the two-flux problem where the flux is  given by  \eqref{two_flux}. It is readily verified that
$\bar{A}(u,u,x_j,x_j) = A(u,x_j)$ and that $\bar{A}(u,v,x_j,x_{j+1})$ is nondecreasing (nonincreasing) as a function of $u$ ($v$).
%

Consider $\Psi(\cdot,\cdot)$ as in \eqref{Psi}. Let 
\begin{equation}\label{define_zjn}
z_j^n = \Psi(u_j^n,x_j), \quad z^{\D}(x,t) = \sum_{n=0}^N \sumj \chi_j(x) \chi^n(t) z_j^n.
\end{equation}
We obtain compactness for $\{u^{\D}\}$ via the singular mapping technique, which
consists of first proving compactness for the sequence $\{z^{\D}\}$, and then
observing that convergence of the original sequence $\{u^{\D} \}$ follows from the fact that
$u \mapsto \Psi(u,x)$ has a continuous inverse.

For our analysis we will assume that $u_0$ has compact support and $u_0 \in \BV(\R)$.
Thus all of the sums in what follows are finite.
We will show in section \ref{entropy_section_A} that the solution we obtain as a limit of numerical approximations satisfies the adapted entropy inequality \eqref{AEC}. Using \eqref{contraction}, the resulting existence theorem is then extended to
the case of $u_0 \in  L^{\infty}(\R)$ via approximations to $u_0$ that are in $\BV$ and have compact support.

Let
\begin{equation}\label{A_2}
\bar{\alpha} = \sup_{x \in \R} A(u_0(x),x).
\end{equation}
By Assumption \descref{H1}{H-1}, and since $\norm{u_0}_{\infty}<\infty$ (which follows from $u_0 \in \BV(\R)$), $\bar{\alpha} < \infty$.
Define $k_{\bar{\alpha}}^{\pm}(x)$ via the equations
\begin{equation}\label{define_k}
\begin{split}
&A(k_{\bar{\alpha}}^{-}(x),x) = \bar{\alpha}, \quad k_{\bar{\alpha}}^{-}(x) \le u_M(x),\\
&A(k_{\bar{\alpha}}^{+}(x),x) = \bar{\alpha}, \quad k_{\bar{\alpha}}^{+}(x) \ge u_M(x).
\end{split}
\end{equation}
\begin{lemma}\label{k_bounded}
The following bounds are satisfied: 
\begin{equation}
\sup_{x \in \R}k_{\bar{\alpha}}^{\pm}(x) < \infty.
\end{equation}
\end{lemma}

\begin{proof}
By definition, $u_M(x) \le k^+_{\bar{\alpha}}(x)$.
On the other hand, by \eqref{uniform_unimodal} and \eqref{define_k}, we have
\begin{equation}\label{gamma_bound}
\textrm{$\gamma(k^+_{\bar{\alpha}}(x)-u_M(x)) \le \bar{\alpha}$ for all $x \in \R$}.
\end{equation}
By Assumption \descref{H3}{H-3}, $\gamma^{-1}$ is defined on
$[0,\infty)$. Applying $\gamma^{-1}$ to both sides of \eqref{gamma_bound}
yields
\begin{equation}
\textrm{$k^+_{\bar{\alpha}}(x)-u_M(x) \le \gamma^{-1}(\bar{\alpha})$ for all $x \in \R$}.
\end{equation}
Thus
\begin{equation}\label{k_plus_bound}
\textrm{$u_M(x) \le k^+_{\bar{\alpha}}(x) \le u_M(x)+\gamma^{-1}(\bar{\alpha})$ for all $x \in \R$}.
\end{equation}
The desired bound for $k^+_{\bar{\alpha}}(x)$ then follows from \eqref{k_plus_bound}, 
along with the fact that $u_M \in L^{\infty}(\R)$ (which follows from $u_M \in \BV(\R)$).
The proof for $k^-_{\bar{\alpha}}(x)$ is similar.
\end{proof}

Let 
\begin{eqnarray}\label{def_B_L1}
\begin{array}{lll}
\mM = \max\left(\sup_{x \in \R} \abs{k_{\bar{\alpha}}^{-}(x)}, 
\sup_{x \in \R} \abs{k_{\bar{\alpha}}^{+}(x)}\right), \quad \bar{R}=\sup\{R(u);\,\abs{u}\leq \mM\},\\
L =\sup\{\abs{\partial_u A(u,x)}:\abs{u} \le \mM, x \in \R \}.
\end{array}
\end{eqnarray}
Note that by Assumption \descref{H1}{H-1}, $L< \infty$. Since $R$ is continuous we have $\bar{R}<\f$.
Also, by \eqref{define_k} we have $k^-_{\bar{\alpha}}(x) \le u_M(x) \le k^+_{\bar{\alpha}}(x)$ for all $x \in \R$,
implying that $\norm{u_M}_{\infty} \le \mM$.

\begin{lemma}\label{lemma_flux_estimates}
	The numerical flux $\bar{A}$ satisfies the following continuity estimates:
	\begin{equation}\label{flux_estimate}
	\begin{split}
	&\abs{\bar{A}(\hat{u},v,x,y) - \bar{A}(u,v,x,y)} \le L \abs{\hat{u}-u},\\
	&\abs{\bar{A}(u,\hat{v},x,y) - \bar{A}(u,v,x,y)} \le L\abs{\hat{v}-v},\\
	&\abs{\bar{A}(u,v,\hat{x},y) - \bar{A}(u,v,x,y)} \le R(\max(u,u_M(\hat{x})))\abs{a(\hat{x})-a(x)}
	+ L\abs{u_M(\hat{x})-u_M(x)},\\
	&\abs{\bar{A}(u,v,x,\hat{y}) - \bar{A}(u,v,x,y)} \le R(\min(v,u_M(\hat{y})))\abs{a(\hat{y})-a(y)}
	+ L\abs{u_M(\hat{y})-u_M(y)},
	\end{split}
	\end{equation}
	for $u,\hat{u},v,\hat{v}\in[-\mM,\mM]$ and $x,\hat{x},y,\hat{y}\in\re$. 
\end{lemma}

\begin{proof}
	These inequalities follow from the definition of $\bar{A}$ along with 
	\begin{equation}\label{abc_inequality}
	\abs{\max(a,b)-\max(c,b)} \le \abs{a-c}, \quad
	\abs{\min(a,b)-\min(c,b)} \le \abs{a-c}.
	\end{equation}
	More specifically, from \eqref{def_bar_A_direct} and \eqref{abc_inequality} we have 
	\begin{eqnarray*}
		\abs{\bar{A}(\hat{u},v,x,y)-\bar{A}(u,v,x,y)}&\leq& \abs{A(\max(\hat{u},u_M(x)),x)-A(\max(u,u_M(x)),x)}\\
		&\leq&L\abs{\max(\hat{u},u_M(x))-\max(u,u_M(x))}\\
		&\leq&L\abs{\hat{u}-u}.
	\end{eqnarray*}
	The second inequality in \eqref{flux_estimate} can be proven in a similar manner. By using definition 		\eqref{def_bar_A_direct} and inequalities in \eqref{abc_inequality} we have
	\begin{eqnarray*}
		\abs{\bar{A}(u,v,\hat{x},y)-\bar{A}(u,v,x,y)}&\leq&\abs{A(\max(u,u_M(\hat{x})),\hat{x})-A(\max(u,u_M(x)),x)}\\
		&\leq&\abs{A(\max(u,u_M(\hat{x})),x)-A(\max(u,u_M(x)),x)}\\
		&+&\abs{A(\max(u,u_M(\hat{x})),\hat{x})-A(\max(u,u_M(\hat{x})),x)}\\
		&\leq&L\abs{\max(u,u_M(\hat{x}))-\max(u,u_M(x))}\\
		&+&R(\max(u,u_M(\hat{x})))\abs{a(\hat{x})-a(x)}\\
		&\leq&L\abs{u_M(\hat{x})-u_M(x)}+R(\max(u,u_M(\hat{x})))\abs{a(\hat{x})-a(x)}.
	\end{eqnarray*} 
	The last inequality in \eqref{flux_estimate} can be shown in a similar way.
\end{proof}

\begin{lemma}\label{lemma:Psi_lipschitz}
Let $u\in[-\mM,\mM]$ and $x,y\in\re$. Then we have
\begin{equation}
\abs{\Psi(u,x)-\Psi(u,y)}\leq L\abs{u_M(x)-u_M(y)}+R(u)\abs{a(x)-a(y)}.
\end{equation}	
\end{lemma}
\begin{proof}
	We start with the observation that
	\begin{equation*}
	\Psi(u,x)-\Psi(u,y)=\left\{\begin{array}{lll}
	-A(u,x)+A(u,y)&\mbox{ if }u_M(x)\geq u\mbox{ and }u_M(y)\geq u,&\\
	A(u,x)+A(u,y)&\mbox{ if }u_M(x)<u\mbox{ and }u_M(y)\geq u,&\\
	-A(u,x)-A(u,y)&\mbox{ if }u_M(x)\geq u\mbox{ and }u_M(y)<u,&\\
	A(u,x)-A(u,y)&\mbox{ if }u_M(x)<u\mbox{ and }u_M(y)<u.&\\
	\end{array}\right.
	\end{equation*}
	When $u_M(x)<u\leq u_M(y)$ we have
	\begin{eqnarray*}
		\Psi(u,x)-\Psi(u,y)&=&A(u,x)+A(u,y)\\
		&\leq& \abs{A(u,x)-A(u_M(x),x)}+\abs{A(u,y)-A(u_M(y),y)}\\
		&\leq& L\abs{u_M(x)-u_M(y)}.
	\end{eqnarray*}
	Similarly, when $u_M(y)<u\leq u_M(x)$ we have
	\begin{equation*}
	\Psi(u,x)-\Psi(u,y)\leq L\abs{u_M(x)-u_M(y)}.
	\end{equation*}
	In the other cases we can estimate directly and get
	\begin{equation*}
	\Psi(u,x)-\Psi(u,y)\le R(u)\abs{a(x)-a(y)}.
	\end{equation*}
	By symmetry we have
	\begin{equation*}
	\abs{\Psi(u,x)-\Psi(u,y)}\leq R(u)\abs{a(x)-a(y)}+L\abs{u_M(x)-u_M(y)}.
	\end{equation*}
\end{proof}

\begin{lemma}\label{lemma:k_stationary_A}
	The grid functions $\{ k_{\bar{\alpha}}^{-}(x_j)\}_{j \in \Z}$ 
	and $\{ k_{\bar{\alpha}}^{+}(x_j)\}_{j \in \Z}$
	are stationary solutions of the difference scheme.
\end{lemma}

\begin{proof}
	We will prove the lemma for $\{ k_{\bar{\alpha}}^{+}(x_j)\}$. The  proof
	for $\{ k_{\bar{\alpha}}^{-}(x_j)\}$ is similar.
	It suffices to show that 
	\begin{equation*}
	\bar{A}(k_{\bar{\alpha}}^{+}(x_j),k_{\bar{\alpha}}^{+}(x_{j+1}),x_j,x_{j+1}) = \bar{\alpha}, \quad j \in \Z.
	\end{equation*}
	By definition
	\begin{equation*}
	k_{\bar{\alpha}}^{+}(x_j) \ge u_M(x_j), \quad k_{\bar{\alpha}}^{+}(x_{j+1}) \ge u_M(x_{j+1}).
	\end{equation*}
	Thus, referring to \eqref{def_bar_A_direct},
	\begin{equation*}
	\begin{split}
	\bar{A}(k_{\bar{\alpha}}^{+}(x_j),k_{\bar{\alpha}}^{+}(x_{j+1}),x_j,x_{j+1})
	&= \max \left\{A(k_{\bar{\alpha}}^{+}(x_j),x_j) , A(u_M(x_{j+1}),x_{j+1})  \right\}\\
	&= \max \left\{\bar{\alpha} , 0)  \right\} = \bar{\alpha}.
	\end{split}
	\end{equation*}
	Recalling the formula for the scheme \eqref{scheme_A}, it is clear from the above that 
	$\{ k_{\bar{\alpha}}^{-}(x_j)\}_{j \in \Z}$ is a
	stationary solution.
\end{proof}

	For the convergence analysis that follows we assume that $\D :=(\Delta x,\Delta t)\rightarrow 0$ with the
	ratio $\lambda = \D t / \D x$ fixed and satisfying the CFL condition
	\begin{equation}\label{CFL}
	\lambda L \le 1.
	\end{equation}

\begin{lemma}\label{lemma_A_scheme_monotone}
	Assume that $\lambda$ is chosen so that the CFL condition \eqref{CFL} holds. 
	The scheme is monotone, meaning that
	if $ \abs{v_j^n}, \abs{w_j^n} \le \mM$ for $ j \in \Z$, then
	\begin{equation*}
	v_j^n \le w_j^n, \quad j \in \Z
	\implies
	v_j^{n+1} \le w_j^{n+1},  \quad j \in \Z.
	\end{equation*}
\end{lemma}

\begin{proof}
	We define $H_j(u,v,w)$ as follows
	\begin{equation}\label{Hj}
	H_j(u,v,w):=v-\la\left(\bar{A}(v,w,x_j,x_{j+1})-\bar{A}(u,v,x_{j-1},x_j)\right).
	\end{equation}
	We show that $H_j$ is nondecreasing in each variable. Note that from definition (\ref{def_bar_A_direct}) it is clear that $\bar{A}(\cdot,\cdot,x_j,x_{j+1})$ is nondecreasing in the first variable and nonincreasing in the second variable. Therefore, from \eqref{Hj} we have
	\begin{eqnarray}
	H_j(u_1,v,w)\leq H_j(u_2,v,w)&\mbox{ for }&u_1\leq u_2\nonumber,\\
	H_j(u,v,w_1)\leq H_j(u,v,w_2)&\mbox{ for }&w_1\leq w_2\nonumber.
	\end{eqnarray}
	Next we define 
	\begin{eqnarray}
	{u}^*\leq u_M(x_j)\mbox{ such that }A(u,x_{j-1})=A(u^*,x_{j})\mbox{ when }u\geq u_M(x_{j-1})\nonumber,\\
	{w}_*\geq u_M(x_j)\mbox{ such that }A(w,x_{j+1})=A(w_*,x_{j})\mbox{ when }w\leq u_M(x_{j+1})\nonumber.
	\end{eqnarray}
	For $v_1\leq v_2$ we denote $I_1=\bar{A}(u,v_1,x_{j-1},x_{j})-\bar{A}(u,v_2,x_{j-1},x_{j})$ and $I_2=\bar{A}(v_1,w,x_{j},x_{j+1})-\bar{A}(v_2,w,x_{j},x_{j+1})$ and $I=I_1-I_2$. From (\ref{def_bar_A_direct}) we have the following:
	\begin{eqnarray}
	I_1&=&\left\{\begin{array}{lll}\label{I1}
	A(v_1,x_j)-A(u^*,x_j)&\mbox{ if }u\geq u_M(x_{j-1})\mbox{ and }v_1\leq u^*\leq v_2,&\\
	A(v_1,x_j)-A(v_2,x_j)&\mbox{ if }u\geq u_M(x_{j-1})\mbox{ and }v_1\leq v_2\leq u^*,&\\
	A(v_1,x_j)-A(v_2,x_j)&\mbox{ if }u\leq u_M(x_{j-1})\mbox{ and }v_1\leq v_2\leq u_M(x_j),&\\
	A(v_1,x_j)&\mbox{ if }u\leq u_M(x_{j-1})\mbox{ and }v_1\leq u_M(x_j) \leq v_2,&\\	
	0&\mbox{ otherwise.}&
	\end{array}\right.\\
	I_2&=&\left\{\begin{array}{lll}\label{I2}
	A(w_*,x_j)-A(v_2,x_j)&\mbox{ if }w\leq u_M(x_{j+1})\mbox{ and }v_1\leq w_*\leq v_2,&\\
	A(v_1,x_j)-A(v_2,x_j)&\mbox{ if }w\leq u_M(x_{j+1})\mbox{ and }w_*\leq v_1\leq v_2,&\\
	A(v_1,x_j)-A(v_2,x_j)&\mbox{ if }w\geq u_M(x_{j+1})\mbox{ and }u_M(x_j)\leq v_1\leq v_2,&\\
	-A(v_2,x_j)&\mbox{ if }w\geq u_M(x_{j+1})\mbox{ and }v_1\leq u_M(x_j) \leq v_2,&\\	
	0&\mbox{ otherwise.}&
	\end{array}\right.
	\end{eqnarray}
	From (\ref{I1}) and (\ref{I2}) it follows that $I=-I_2$ if $v_1\geq u_M(x_j)$ and $I=I_1$ if $v_2\leq u_M(x_j)$. In both the cases we have $|I|\leq L|v_1-v_2|$. When $v_1\leq u_M(x_j)\leq v_2$ we have
	\begin{eqnarray}
	|I|\leq |I_1|+|I_2|
	&\leq& L( |v_1-u_M(x_j)|+|v_2-u_M(x_j)|)=L|v_1-v_2|.
	\end{eqnarray}
	Therefore we have
	\begin{equation*}
	H_j(u,v_1,w)-H_j(u,v_2,w)=v_1-v_2+\la I\leq (1-\la L)(v_1-v_2)\leq0.
	\end{equation*}
	Hence, the proof is completed  by invoking the CFL condition \eqref{CFL}.
\end{proof}

\begin{lemma}\label{lemma:A_u_bound}
	Assume that the CFL condition \eqref{CFL} holds.
	Then,
	\begin{equation}\label{A_}
	\abs{u_j^n} \le \mM, \quad j \in \Z, n \ge 0.
	\end{equation}
\end{lemma}

\begin{proof}
	From \eqref{A_2}, we have 
	\begin{equation}\label{fg_3_m_direct}
	A(u_0(x),x) \le \bar{\alpha}, \quad \forall x \in \R.
	\end{equation}
	Applying the two branches of the inverse function $A^{-1}(\cdot,x)$ to \eqref{fg_3_m_direct}, and using the fact that
	the increasing branch preserves order, while the decreasing branch reverses order, we
	have
	\begin{equation}\label{fg_4_m_direct}
	k_{\bar{\alpha}}^{-}(x) \le u_0(x) \le k_{\bar{\alpha}}^{+}(x), \quad \forall x \in \R.
	\end{equation}
	By evaluation at $x=x_j$, we also have
	\begin{equation}\label{fg_5_m_direct}
	k_{\bar{\alpha}}^{-}(x_j) \le u_j^0 \le k_{\bar{\alpha}}^{+}(x_j),\quad j \in \Z.
	\end{equation}
	Thus $\abs{u_j^0} \le \mM$ for $j \in \Z$.
	It is clear from \eqref{def_B_L1} that also 
	$\abs{k_{\bar{\alpha}}^{\pm}(x_j)} \le \mM$ for $j \in \Z$.
	We apply a single step of the scheme to all three parts of \eqref{fg_5_m_direct}, 
	and due to the bounds 
	$\abs{u_j^0},  \abs{k_{\bar{\alpha}}^{\pm}(x_j)} \le \mM$,  
	the scheme acts in a monotone manner (Lemma~\ref{lemma_A_scheme_monotone}), 
	so that the ordering of
	\eqref{fg_5_m_direct} is preserved. In addition
	each of $\{ k_{\bar{\alpha}}^{-}(x_j)\}$ and $\{ k_{\bar{\alpha}}^{+}(x_j)\}$
	is a stationary solution of the difference scheme, by Lemma~\ref{lemma:k_stationary_A}.
	Thus, after applying the difference scheme, the result is
	\begin{equation}\label{fg_6_m_direct}
	k_{\bar{\alpha}}^{-}(x_j) \le u_j^1 \le k_{\bar{\alpha}}^{+}(x_j),\quad i \in \Z,
	\end{equation}
	implying that \eqref{A_} holds at time level $n=1$.
	The proof is completed by continuing this way from one time step to the next. 
\end{proof}

\begin{lemma}\label{lemma_z_bounded}
The following bound holds for $z_j^n$:
\begin{equation}
\abs{z_j^n} \le 2\mM L, \quad j \in \Z, \quad n \ge 0.
\end{equation}
\end{lemma}
\begin{proof}
	From definition \eqref{Psi} of $\Psi$, \eqref{def_B_L1} and \eqref{A_} we have
	\begin{equation}
	\abs{z^n_j}=\abs{\Psi(u^n_j,x_j)}=\abs{\int\limits_{u_{M}(x_j)}^{u^n_j}\abs{\frac{\pa A}{\pa u}(\theta,x_j)}\,d\theta}\leq L\abs{u^n_j-u_M(x_j)}\leq 2\mM L.\nonumber
	\end{equation}
\end{proof}

\begin{lemma}\label{lemma_time_continuity}
The following time continuity estimate holds for $u_j^n$:
\begin{equation}\label{time_continuity_est_u}
\sumj \abs{u_j^{n+1}-u_j^n} \le 2 \lambda (\bar{R} \TV(a) +L \TV(u_0) + L \TV(u_M)).
\end{equation}
\end{lemma}

\begin{proof}
Since $u_0$ has compact support, the same is true of $\{u_j^n\}$ for $n \ge 0$.
In particular, we have $\sum\nolimits_{j \in \Z}  \abs{u_j^n} < \infty$ for $n \ge 0$. By monotonicity, and
the fact that $\sum\nolimits_{j \in \Z} u_j^{n+1} = \sum\nolimits_{j \in \Z}  u_j^n$, we can invoke the Crandall-Tartar
lemma \cite{Holden_Risebro}, yielding
\begin{equation}
\begin{split}
\sumj \abs{u_j^{n+1}-u_j^n} 
&\le \sumj \abs{u_j^{n}-u_j^{n-1}}\\
&\,\,\, \vdots\\
&\le \sumj \abs{u_j^{1}-u_j^{0}}\\
&\le \lambda \sumj \abs{\bar{A}(u^0_{j},u^0_{j+1},x_{j},x_{j+1})- \bar{A}(u^0_{j-1},u^0_j,x_{j-1},x_j)}.
\end{split}
\end{equation}
The proof will be completed by estimating the last term.

\begin{equation}\label{four_terms_new}
\begin{split}
&\abs{\bar{A}(u^0_j,u^0_{j+1},x_j,x_{j+1})
	-\bar{A}(u^0_{j-1},u^0_{j},x_{j-1},x_{j})}\\
&\le \abs{\bar{A}(u^0_j,u^0_{j+1},x_j,x_{j+1})
	-\bar{A}(u^0_{j},u^0_{j+1},x_{j},x_{j})}\\       
&+  \abs{\bar{A}(u^0_j,u^0_{j+1},x_j,x_{j})
	-\bar{A}(u^0_{j},u^0_{j},x_{j},x_{j})}\\   
&+  \abs{\bar{A}(u^0_j,u^0_{j},x_j,x_{j})
	-\bar{A}(u^0_{j},u^0_{j},x_{j-1},x_{j})}\\    
&+  \abs{\bar{A}(u^0_j,u^0_{j},x_{j-1},x_{j})
	-\bar{A}(u^0_{j-1},u^0_{j},x_{j-1},x_{j})}.          
\end{split}
\end{equation}
Invoking Lemma \ref{lemma_flux_estimates}, we obtain 
\begin{equation}\label{four_terms_1_new}
\begin{split}
&\abs{\bar{A}(u^0_j,u^0_{j+1},x_j,x_{j+1})
	-\bar{A}(u^0_{j},u^0_{j+1},x_{j},x_{j})}
\le \bar{R}\abs{a(x_{j+1}) - a(x_j) } \\
&\qquad \qquad \qquad \qquad \qquad \qquad \qquad \qquad \qquad \quad \,\,\,
+ L \abs{u_M(x_{j+1}) - u_M(x_j)},\\
&\abs{\bar{A}(u^0_j,u^0_{j+1},x_j,x_{j})
	-\bar{A}(u^0_{j},u^0_{j},x_{j},x_{j})}
\le L \abs{u^0_{j+1}-u^0_{j}},\\
&\abs{\bar{A}(u^0_j,u^0_{j},x_j,x_{j})
	-\bar{A}(u^0_{j},u^0_{j},x_{j-1},x_{j})}
\le \bar{R}\abs{a(x_{j}) - a(x_{j-1}) } \\
&\qquad \qquad \qquad \qquad \qquad \qquad \qquad \qquad \quad \quad
+ L \abs{u_M(x_{j}) - u_M(x_{j-1})},\\
&\abs{\bar{A}(u^0_j,u^0_{j},x_{j-1},x_{j})
	-\bar{A}(u^0_{j-1},u^0_{j},x_{j-1},x_{j})}
\le L \abs{u^0_{j}-u^0_{j-1}}.
\end{split}
\end{equation}
Plugging \eqref{four_terms_1_new} into \eqref{four_terms_new}, and then summing over $j \in \Z$, the result is
\begin{equation*}
\begin{split}
\sumj \abs{u_j^{n+1} - u_j^{n}}
&\le 2 \lambda  \bar{R}\sumj \abs{a(x_{j+1}) - a(x_j) } \\
&+ 2 \lambda L\sumj \abs{u^0_{j+1} - u^0_j}
+ 2 \lambda L \sumj \abs{u_M(x_{j+1}) - u_M(x_j)}\\
&\le
2 \lambda (\bar{R} \TV(a) +L \TV(u_0) + L \TV(u_M)).
\end{split}
\end{equation*}
\end{proof}

\begin{lemma}\label{lemma_z_time_continuity}
The following time continuity estimate holds for $z_j^n$:
\begin{equation}\label{time_continuity_est_z}
\sumj \abs{z_j^{n+1} - z_j^n} \le 2 \lambda L (\bar{R} \TV(a) +L \TV(u_0) + L \TV(u_M)),  \quad n \ge 0.
\end{equation}
\end{lemma}

\begin{proof}
The estimate \eqref{time_continuity_est_z} follows directly from time continuity for $\{u_j^n \}$, and Lipschitz continuity of
$\Psi(\cdot,x_j)$. Indeed,
\begin{equation}\label{time_continuity_est_z1}
|z_j^{n+1}-z_j^n|=\abs{\Psi(u^{n+1}_j,x_j)-\Psi(u^n_j,x_j)}=\abs{\int\limits_{u^n_j}^{u^{n+1}_j}\abs{\frac{\pa A}{\pa u}(\theta,x_j)\,d\theta}}\leq L\abs{u^{n+1}_j-u^n_j}.
\end{equation}
Now \eqref{time_continuity_est_z} is immediate from \eqref{time_continuity_est_u} and \eqref{time_continuity_est_z1}.
\end{proof}

We next turn to establishing a spatial variation bound for $z_j^n$.
Define
\begin{equation}
\sigma_+(u,x) = 
\begin{cases}
1, \quad & A_u(u,x) >0,\\
0, \quad & A_u(u,x)\le 0,
\end{cases}
\quad 
\sigma_-(u,x) = 
\begin{cases}
1, \quad & A_u(u,x) <0,\\
0, \quad & A_u(u,x)\ge 0.
\end{cases}
\end{equation}
For the proof of the following lemma we refer the reader 
to Lemma 4.5 of 
\cite{AJG} or Lemma 3.3 of \cite{towers_disc_flux_1}.
\begin{lemma}\label{lemma_integral_ineq_1}
The following inequality holds:
\begin{equation}\label{integral_ineq_1}
\begin{split}
\left(\Psi(u_{j+1}^n,x_j) - \Psi(u_j^n,x_j) \right)_+ 
&\le \sigma_-(u_j^n,x_j) \abs{\bar{A}(u_j^n,u_{j+1}^n,x_j,x_j) - \bar{A}(u_{j-1}^n,u_{j}^n,x_j,x_j)}\\
&+ \sigma_+(u_{j+1}^n,x_{j}) \abs{\bar{A}(u_{j+1}^n,u_{j+2}^n,x_j,x_j) - \bar{A}(u_{j}^n,u_{j+1}^n,x_j,x_j)}.
\end{split}
\end{equation}
\end{lemma}

\begin{lemma}\label{lemma_integral_ineq_2}
Let $\bar{A}^n_{\jph} = \bar{A}(u_j^n,u_{j+1}^n,x_j,x_{j+1})$.  The following inequality holds:
\begin{equation}\label{integral_ineq_2}
\begin{split}
\left(\Psi(u_{j+1}^n,x_{j+1}) - \Psi(u_j^n,x_j) \right)_+ 
&\le \sigma_-(u_j^n,x_j) \abs{\bar{A}^n_{\jph} - \bar{A}^n_{\jmh}}\\
&+ \sigma_+(u_{j+1}^n,x_{j}) \abs{\bar{A}^n_{j+3/2} - \bar{A}^n_{\jph}}\\
&+ \Omega_{\jph}^n,
\end{split}
\end{equation}
where $\sum\nolimits_{j\in \Z} \Omega_{\jph}^n \le C (\TV(a)+\TV(u_M))$ and $C$ is independent of $\D$.
\end{lemma}
\begin{proof}
	By Lemma \ref{lemma:Psi_lipschitz} we have 
	\begin{eqnarray}
	\left(\Psi(u_{j+1}^n,x_{j+1}) - \Psi(u_j^n,x_j) \right)_+ &\leq& \left(\Psi(u_{j+1}^n,x_{j}) - \Psi(u_j^n,x_j) \right)_+\nonumber\\
	& +&\left(\Psi(u_{j+1}^n,x_{j+1}) - \Psi(u_{j+1}^n,x_j) \right)_+ \nonumber\\
	&\leq& \left(\Psi(u_{j+1}^n,x_{j}) - \Psi(u_j^n,x_j) \right)_++L\abs{u_M(x_{j+1})-u_M(x_j)}\nonumber\\
	&+& R(u_{j+1}^n)\abs{a(x_{j+1})-a(x_j)}.
	\end{eqnarray}
	From \eqref{integral_ineq_1} we have
	\begin{eqnarray}
	\left(\Psi(u_{j+1}^n,x_{j+1}) - \Psi(u_j^n,x_j) \right)_+ 
	&\le& \sigma_-(u_j^n,x_j) \abs{\bar{A}(u_j^n,u_{j+1}^n,x_j,x_j) - \bar{A}(u_{j-1}^n,u_{j}^n,x_j,x_j)}\nonumber\\
	&+& \sigma_+(u_{j+1}^n,x_{j}) \abs{\bar{A}(u_{j+1}^n,u_{j+2}^n,x_j,x_j) - \bar{A}(u_{j}^n,u_{j+1}^n,x_j,x_j)}\nonumber\\
		& +&\bar{R}\abs{a(x_{j+1})-a(x_j)}+L\abs{u_M(x_{j+1})-u_M(x_j)}.\label{estimate_on_psi1}
	\end{eqnarray}
	We further modify \eqref{estimate_on_psi1} to get the following
	\begin{equation}\label{estimate_on_psi2}
	\begin{array}{lll}
	& &\left(\Psi(u_{j+1}^n,x_{j+1}) - \Psi(u_j^n,x_j) \right)_+ \\
	&\le& \sigma_-(u_j^n,x_j) \abs{\bar{A}^n_{\jph} - \bar{A}^n_{\jmh}}+ \sigma_+(u_{j+1}^n,x_{j+1}) \abs{\bar{A}^n_{j+3/2} - \bar{A}^n_{\jph}}\\
	& +&\bar{R}\abs{a(x_{j+1})-a(x_j)}+L\abs{u_M(x_{j+1})-u_M(x_j)}\\
	&+& \sigma_-(u_j^n,x_j) \abs{\bar{A}(u_{j}^n,u_{j+1}^n,x_{j},x_{j+1}) - \bar{A}(u_{j}^n,u_{j+1}^n,x_j,x_j)}\\
	&+& \sigma_-(u_j^n,x_j) \abs{\bar{A}(u_{j-1}^n,u_{j}^n,x_{j-1},x_j) - \bar{A}(u_{j-1}^n,u_{j}^n,x_j,x_j)}\\
	&+& \sigma_+(u_{j+1}^n,x_{j}) \abs{\bar{A}(u_{j+1}^n,u_{j+2}^n,x_j,x_j) - \bar{A}(u_{j+1}^n,u_{j+2}^n,x_{j},x_{j+2})}\\
	&+& \sigma_+(u_{j+1}^n,x_{j}) \abs{\bar{A}(u_{j+1}^n,u_{j+2}^n,x_j,x_{j+2}) - \bar{A}(u_{j+1}^n,u_{j+2}^n,x_{j+1},x_{j+2})}\\
	&+& \sigma_+(u_{j+1}^n,x_{j}) \abs{\bar{A}(u_{j}^n,u_{j+1}^n,x_j,x_{j+1}) - \bar{A}(u_{j}^n,u_{j+1}^n,x_j,x_{j})}.
	\end{array}
	\end{equation}
	Next we apply Lemma \ref{lemma_flux_estimates} to bound the last five terms  of \eqref{estimate_on_psi2} by
	\begin{equation}\label{estimate_on_psi3}
	\begin{array}{lll}
	\bar{R}\abs{a(x_j)-a(x_{j-1})}+3\bar{R}\abs{a(x_{j+1})-a(x_{j})}+\bar{R}\abs{a(x_{j})-a(x_{j+2})}\\
	+L\abs{u_M(x_j)-u_M(x_{j-1})}+3L\abs{u_M(x_j)-u_M(x_{j+1})}+L\abs{u_M(x_{j})-u_M(x_{j+2})}.
	\end{array}
	\end{equation}
	Combining \eqref{estimate_on_psi2} and \eqref{estimate_on_psi3} we get \eqref{integral_ineq_2} with 
	\begin{eqnarray*}
	\Omega_{\jph}&:=&\bar{R}\abs{a(x_j)-a(x_{j-1})}+4\bar{R}\abs{a(x_{j+1})-a(x_{j})}+\bar{R}\abs{a(x_{j+2})-a(x_{j})}\\
	&+&L\abs{u_M(x_j)-u_M(x_{j-1})}+4L\abs{u_M(x_{j+1})-u_M(x_{j})}+L\abs{u_M(x_{j+2})-u_M(x_{j})}.
	\end{eqnarray*}
\end{proof}

\begin{lemma}\label{lemma_Psi_limits}
There are real numbers $P^{\pm}$ such that
\begin{equation}\label{P+-}
\lim_{x \rightarrow \pm \infty} \Psi(0,x) = P^{\pm}.
\end{equation}
\end{lemma}

\begin{proof}
	By Lemma \ref{lemma:Psi_lipschitz} we have
	\begin{equation}\label{Psi_0x_BV}
	\abs{\Psi(0,x)-\Psi(0,y)}\leq \bar{R}\abs{a(x)-a(y)}+L\abs{u_M(x)-u_M(y)}.
	\end{equation}
	Along with Assumptions \descref{H2}{H-2} and \descref{H4}{H-4} the estimate \eqref{Psi_0x_BV} yields $x\mapsto\Psi(0,x)$ is a $BV$ function on $\re$. Then \eqref{P+-} follows by a standard property of $BV$ functions.
\end{proof}

\begin{lemma}\label{lemma_tv_bound_z}
The following spatial variation bound holds for $n \ge 0$:
\begin{equation}\label{tv_bound_z}
\sumj \abs{z^n_{j+1}-z^n_j} \le C,
\end{equation}
where $C$ is a $\D$-independent constant.
\end{lemma}

\begin{proof}
From Lemma~\ref{lemma_integral_ineq_2} we find that
\begin{equation}
\begin{split}
\sumj (z^n_{j+1}-z^n_j)_+ 
&\le \sumj \sigma_-(u_j^n,x_j) \abs{\bar{A}^n_{\jph} - \bar{A}^n_{\jmh}}
+ \sumj  \sigma_+(u_{j+1}^n,x_{j}) \abs{\bar{A}^n_{j+3/2} - \bar{A}^n_{\jph}}\\
&+ \sumj \Omega_{\jph}^n\\
&\le 2 \sumj \abs{\bar{A}^n_{\jph} - \bar{A}^n_{\jmh}} + \sumj \Omega_{\jph}^n\\
&= {2 \over \lambda} \sumj \abs{u_j^{n+1}-u_j^n} + \sumj \Omega_{\jph}^n.
\end{split}
\end{equation}
Invoking Lemma~\ref{lemma_time_continuity} and the fact that
$\sum\nolimits_{j \in \Z} \Omega_{\jph}^n \le C (\TV(a)+\TV(u_M))$, the result is
\begin{equation}\label{sum_dz_1}
\sumj (z^n_{j+1}-z^n_j)_+  \le C_2,
\end{equation}
for some $\D$-independent constant $C_2$.

As a result of Lemma~\ref{lemma_Psi_limits}, along with the fact that $\{ u_j^n\}$ has compact
support,
\begin{equation}\label{sum_dz_2}
\sumj (z^n_{j+1}-z^n_j) =P^+ - P^-.  
\end{equation}
We also have
\begin{equation}
\sumj (z^n_{j+1}-z^n_j) =\sumj (z^n_{j+1}-z^n_j)_+ + \sumj (z^n_{j+1}-z^n_j)_-,
\end{equation}
implying that
\begin{equation}\label{sum_dz_3}
-\sumj (z^n_{j+1}-z^n_j)_- = \sumj (z^n_{j+1}-z^n_j)_+ + P^- - P^+.
\end{equation}
From \eqref{sum_dz_1} and \eqref{sum_dz_3} it follows that
\begin{equation}
\sumj \abs{z^n_{j+1}-z^n_j} =  \sumj (z^n_{j+1}-z^n_j)_+ - \sumj (z^n_{j+1}-z^n_j)_-
\le 2 C_2 + P^- - P^+.
\end{equation}
\end{proof}

\begin{lemma}\label{lemma_convergence}
The approximations $u^{\D}$ converge as $\D \rightarrow 0$, modulo extraction of a subsequence, in $L^1_{\loc}(\Pi_T)$ and pointwise
a.e. in $\Pi_T$ to a function $u \in L^{\infty}(\Pi_T) \cap  C([0,T]:L^1_{\loc}(\R))$.
\end{lemma}

\begin{proof}
From Lemma \ref{lemma_z_bounded}, \ref{lemma_z_time_continuity} and \ref{lemma_tv_bound_z} we have for some subsequence, and some 
$z \in L^1(\Pi_T) \cap L^{\infty}(\Pi_T)$, 
$z^{\D} \rightarrow z$ in $L^1(\Pi_T)$ 
and pointwise a.e. 
Define $u(x,t) = \Psi^{-1}(z(x,t),x)$.
We have
$u_j^n = \Psi^{-1}(z_j^n,x_j)$, or 
\begin{equation}
\textrm{$u^{\D}(x,t) = \Psi^{-1}(z^{\D}(x,t),x)$ for $(x,t) \in  \Pi_T$}.
\end{equation}
Fixing a point $(x,t) \in \Pi_T$ where $z^{\D}(x,t) \rightarrow z(x,t)$, and using the continuity
of $\zeta \mapsto \Psi^{-1}(\zeta,x)$ for each fixed $x \in \R$, we get 
\begin{equation}\label{convergence_u_De}
u^{\D}(x,t) \rightarrow u(x,t).
\end{equation}
Thus $u^{\D} \rightarrow u$ pointwise a.e. in $\Pi_T$. Since $u^{\D}$ is bounded in $\Pi_T$ independently of
$\D$ in $\Pi_T$, we also have $u^{\D} \rightarrow u$ in $L^1_{\loc}(\Pi_T)$, by the dominated convergence theorem. 
In fact, due to the time continuity estimate \eqref{time_continuity_est_u}, 
we also have $u \in C([0,T]:L^1_{\loc}(\R))$.
\end{proof}

\section{Entropy inequality and proof of Theorem~\ref{theorem1}}\label{entropy_section_A}
In this section we show that $u$ satisfies adapted entropy inequality \eqref{AEC}, the remaining ingredient required for the
proof of Theorem \ref{theorem1} .

\begin{lemma}\lab{entropy_lemma1_A}
	We have the following discrete entropy inequalities:
	\begin{equation}\label{ent_discrete_A}
	\abs{u^{n+1}_j- k^{\pm}_{\alpha,j}} 
	\le \abs{u_j^n - k^{\pm}_{\alpha,j}}
	- \lambda (\mF^n_{\jph} - \mF^n_{\jmh}),\mbox{ for all }j\in\Z
	\end{equation}
	where
	\begin{equation*}
	\mathcal{F}^n_{\jph} = \bar{A}(u_j^n \vee k^{\pm}_{\alpha,j},u_{j+1}^n \vee k^{\pm}_{\alpha,j+1},x_j,x_{j+1})	
	                                 -  \bar{A}(u_j^n \wedge k^{\pm}_{\alpha,j},u_{j+1}^n \wedge k^{\pm}_{\alpha,j+1},x_h,x_{j+1}).
	\end{equation*}
\end{lemma}
\begin{proof}
	The proof is a slightly generalized version of a now classical argument found in
	\cite{CranMaj:Monoton} or \cite{Holden_Risebro}.	
	Denote the grid function $\{u_j^n\}_{j \in \Z}$ by $U^n$, and
	write the scheme defined by
	\eqref{scheme_A}	as $U^{n+1} = \Gamma(U^n)$, i.e.,
	$\Gamma(\cdot)$ is the operator that advances the solution from
	time level $n$ to $n+1$.
	Let $K^{\pm}_{\alpha} = \{k_{\alpha}^{\pm}(x_j)\}_{j \in \Z}$. Since the scheme is
	monotone, we have
	\begin{equation}\label{ent_est1}
	\Gamma(K^{\pm}_{\alpha}) \vee \Gamma(U^n) \le \Gamma(K^{\pm}_{\alpha} \vee U^n), \quad
	\Gamma(K^{\pm}_{\alpha}) \wedge \Gamma(U^n) \ge \Gamma(K^{\pm}_{\alpha} \wedge U^n).
	\end{equation}
	Using the fact that 
	$\Gamma(K^{\pm}_{\alpha}) = K^{\pm}_{\alpha}$, it follows from
	\eqref{ent_est1} that
	\begin{equation}\label{ent_est2}
	U^{n+1} \vee K^{\pm}_{\alpha} - U^{n+1} \wedge K^{\pm}_{\alpha}
	\le  \Gamma(K^{\pm}_{\alpha} \vee U^n)
	-  \Gamma(K^{\pm}_{\alpha} \wedge U^n).
	\end{equation}
	The discrete entropy inequality \eqref{ent_discrete_A} then follows from \eqref{ent_est2}, using
	the definition of $\Gamma(\cdot)$ in terms of \eqref{scheme_A} along
	with the identity $a \vee b - a \wedge b = \abs{a-b}$.
\end{proof}

\begin{lemma}\label{lemma_k_converge}
Let
\begin{equation}
k_{\alpha}^{\pm,\D}(x) = \sumj \chi_j(x) k_{\alpha,j}^{\pm}.
\end{equation}
Then
\begin{equation}
\textrm{$k_{\alpha}^{\pm,\D}(x) \rightarrow k_{\alpha}^{\pm}(x)$
in $L^1_{\loc}(\R)$ and pointwise a.e. in $\R$.}
\end{equation}
\end{lemma}
\begin{proof}
	We first show that $\Psi(k^{\pm,\D}_\al(\cdot),\cdot)\rr\Psi(k^{\pm}_\al(\cdot),\cdot)$ as $\D\rr0$. Observe that 
	\begin{equation*}
	\Psi(k^{\pm,\D}(x),x)= \sumj \chi_j(x) \Psi(k_{\alpha,j}^{\pm},x).
	\end{equation*}
	This yields
	\begin{eqnarray*}
	\abs{\Psi(k^{\pm,\D}_{\al}(x),x)-\Psi(k^{\pm}_\al(x),x)}
		&\leq&\sumj\chi_j(x)\abs{\Psi(k^{\pm}_{\al,j},x)-\Psi(k^{\pm}_\al(x),x)}\\
		&\leq&\sumj\chi_j(x)\abs{\Psi(k^{\pm}_{\al,j},x)-\Psi(k^{\pm}_{\al,j},x_j)}\\
		&+&\sumj\chi_j(x)\abs{\Psi(k^{\pm}_{\al,j},x_j)-\Psi(k^{\pm}_\al(x),x)}\\
		&=&\sumj\chi_j(x)\abs{\Psi(k^{\pm}_{\al,j},x)-\Psi(k^{\pm}_{\al,j},x_j)}\\
		&+&\sumj\chi_j(x)\abs{A(k^{\pm}_{\al,j},x_j)-A(k^{\pm}_\al(x),x)}\\
				&=&\sumj\chi_j(x)\abs{\Psi(k^{\pm}_{\al,j},x)-\Psi(k^{\pm}_{\al,j},x_j)}.
	\end{eqnarray*}
	By virtue of Assumption \descref{H2}{H-2} we obtain
	\begin{eqnarray*}
	\abs{\Psi(k^{\pm,\D}_{\al}(x),x)-\Psi(k^{\pm}_\al(x),x)}
	&\leq&\sumj\chi_j(x)R(k^{\pm}_{\al,j})\abs{a(x)-a(x_j)}\\
	&\leq&\bar{R}\abs{a(x)-\sumj\chi_j(x)a(x_j)}.
	\end{eqnarray*}
	From \eqref{convergence_a(x)}, we have $\Psi(k^{\pm,\D}_{\al}(x),x)\rr\Psi(k^{\pm}_\al(x),x)$ as $\D\rr0$ for a.e. $x\in\re$ . By using continuity of $\zeta\mapsto\Psi^{-1}(\zeta,x)$ for each fixed $x\in\re$ we have $k^{\pm,\D}_{\al}(x)\rr k^{\pm}_\al(x)$ as $\D\rr0$ for a.e. $x\in\re$.
\end{proof}

\begin{lemma}\label{lemma_a_converge}
Define
$a^{\D}(x):=\sumj\chi_j(x)a(x_j)$.
As $\D \rightarrow 0$, $a^{\D} \rightarrow a$ in $L^1_{\loc}(\R)$.
\end{lemma}
\begin{proof}
Suppose 
\begin{eqnarray*}
\overline{m}_j&:=&\sup\{a(x);\,x\in I_j\},\\
\underline{m}_j&:=&\inf\{a(x);\,x\in I_j\}.
\end{eqnarray*}
Then for any $\si>0$, we have
\begin{equation}\label{convergence_a(x)}
\int\limits_{[-\si,\si]}\abs{a(x)-a^{\D}(x)}\,dx\leq\sumj\int\limits_{I_j}(\overline{m}_j-\underline{m}_j)\,dx\leq \D x\,\TV(a)\rr0\mbox{ as }\D x\rr0.
\end{equation}
\end{proof}

\begin{lemma}\label{lemma_lim_kruzkov}
The (subsequential) limit $u$ guaranteed by Lemma~\ref{lemma_convergence} satisfies the adapted entropy inequalities \eqref{kruzkov_type}.
\end{lemma}

\begin{proof}
Fix $\alpha \ge 0$. Define
\begin{equation}
v_j^n := u_j^n \vee k^{\pm}_{\alpha,j}, \quad w_j^n := u_j^n \wedge k^{\pm}_{\alpha,j},
\end{equation}
and
\begin{equation}
v^{\D}(x,t) := \sum_{n=0}^N \sumj \chi_j(x) \chi^n(t) v_j^n, \quad 
w^{\D}(x,t) := \sum_{n=0}^N \sumj \chi_j(x) \chi^n(t) w_j^n.
\end{equation}
We can rewrite $v^{\D}$ and $w^{\D}$ as follows 
\begin{eqnarray*}
	2v^{\D}(x,t)=u^{\D}(x,t)+k^{\pm,\D}_{\al}(x,t)+\abs{u^{\D}(x,t)-k^{\pm,\D}_{\al}(x,t)},\\
	2w^{\D}(x,t)=u^{\D}(x,t)+k^{\pm,\D}_{\al}(x,t)-\abs{u^{\D}(x,t)-k^{\pm,\D}_{\al}(x,t)}.
\end{eqnarray*}
Invoking the convergence results for $u^{\D}$ and for $k^{\pm,\D}_{\alpha}$, we have
\begin{equation}
v^{\D} \rightarrow  \underbrace{u \vee k^{\pm}_{\alpha}}_{=:v}, \quad 
w^{\D} \rightarrow  \underbrace{u \wedge k^{\pm}_{\alpha}}_{=:w}
\end{equation}
pointwise a.e. and in $L^1_{\loc}(\Pi_T)$.  
	From Lemma \ref{entropy_lemma1_A} we have
	\begin{equation}\label{ent_discrete_j}
	\abs{u^{n+1}_j - k^{\pm}_{\alpha,j}} 
	\le \abs{u_j^n - k^{\pm}_{\alpha,j}}
	- \lambda (\mF^n_{\jph} - \mF^n_{\jmh}),\mbox{ for all }i\in\Z
	\end{equation}
	where
	\begin{equation*}
	\mathcal{F}^n_{\jph} = \bar{A}(u_j^n \vee k^{\pm}_{\alpha,j},u_{j+1}^n \vee k^{\pm}_{\alpha,j+1},x_j,x_{j+1})	
	-  \bar{A}(u_j^n \wedge k^{\pm}_{\alpha,j},u_{j+1}^n \wedge k^{\pm}_{\alpha,j+1},x_j,x_{j+1}).
	\end{equation*}
Let $0 \le \phi(x,t) \in C_0^1(\R \times (0,T))$ be a test function, and define $\phi_j^n = \phi(x_j,t^n)$.
As in the proof of the Lax-Wendroff theorem, we multiply \eqref{ent_discrete_j} by $\phi_j^n \D x$, and then
sum by parts to arrive at
\begin{equation}\label{ent12}
\begin{split}
\D x \D t \sum_{n=0}^N \sumj \abs{u_j^{n+1} - k^{\pm}_{\alpha,j}} (\phi_j^{n+1} -\phi_j^n)/\D t
&+\D x \D t \sum_{n=0}^N \sumj \mathcal{F}^n_{j+1/2} (\phi_{j+1}^{n} -\phi_j^n)/\D x\\
&+ \D x \sumj \abs{u_j^0 -  k^{\pm}_{\alpha,j}}\phi_j^0 \ge 0.
\end{split}
\end{equation}
The first and third sums on the left side of \eqref{ent12} converge to
$\int_{\re_+}\int_{\re}\frac{\pa\phi}{\pa t}|u(x,t)-k^{\pm}_{\al}(x)|\,dxdt$
and 
$\int_{\re} \abs{u_0(x) -k^{\pm}_{\al}(x)}\phi(x,0) \,dx$, respectively.
The crucial part of the argument is to prove convergence of the second sum on left hand side of \eqref{ent12}. It suffices to prove that
\begin{equation}\label{Int_1}
\mathcal{I}_1:=\D x \D t \sum_{n=0}^N \sumj \bar{A}(v_j^n,v_{j+1}^n,x_j,x_{j+1}) (\phi_{j+1}^{n} -\phi_j^n)/\D x \rightarrow
\int_0^T \int_{\R} A(v,x) \phi_x \, dx \, dt,
\end{equation}
and that
\begin{equation}\label{Int_2}
\mathcal{I}_2:=\D x \D t \sum_{n=0}^N \sumj \bar{A}(w_j^n,w_{j+1}^n,x_j,x_{j+1}) (\phi_{j+1}^{n} -\phi_j^n)/\D x \rightarrow
\int_0^T \int_{\R} A(w,x) \phi_x \, dx \, dt.
\end{equation}
We will prove \eqref{Int_1}. The proof of \eqref{Int_2} is similar. We start with the following identity:
\begin{equation}
\begin{split}
\bar{A}(v_j^n,v_{j+1}^n,x_j,x_{j+1}) - \bar{A}(v_j^n,v_{j}^n,x_j,x_{j})
&=\bar{A}(v_j^n,v_{j+1}^n,x_j,x_{j+1}) - \bar{A}(v_j^n,v_{j}^n,x_j,x_{j+1})\\
&+\bar{A}(v_j^n,v_{j}^n,x_j,x_{j+1}) - \bar{A}(v_j^n,v_{j}^n,x_j,x_{j}).
\end{split}
\end{equation}
Taking absolute values, using $\bar{A}(v_j^n,v_{j}^n,x_j,x_{j}) = A(v_j^n,x_j)$, and \eqref{flux_estimate},
we have
\begin{equation}
\begin{split}
\abs{\bar{A}(v_j^n,v_{j+1}^n,x_j,x_{j+1}) - A(v_j^n,x_j)}
&\le L \abs{v_{j+1}^n-v_j^n}+ L \abs{u_M(x_{j+1}) - u_M(x_j)}\\
&+R(\min(v_j^n,u_M(x_j)))\abs{a(x_{j+1})-a(x_j)}.
\end{split}
\end{equation}
Thus, with the abbreviation $\rho_j^n := (\phi_{j+1}^{n} -\phi_j^n)/\D x$,
\begin{equation}\label{Int_3}
\begin{split}
&\D x \D t \sum_{n=0}^N \sumj \abs{\bar{A}(v_j^n,v_{j+1}^n,x_j,x_{j+1}) - A(v_j^n,x_j)} \rho_j^n
\le L \underbrace{\D x \D t \sum_{n=0}^N \sumj \abs{v_{j+1}^n-v_j^n} \rho_j^n}_{S_1} \\
&+ \underbrace{\bar{R}\D x \D t \sum_{n=0}^N \sumj \abs{a(x_{j+1})-a(x_j)} \rho_j^n}_{S_2}
+ L \underbrace{\D x \D t \sum_{n=0}^N \sumj \abs{u_M(x_{j+1}) - u_M(x_j)} \rho_j^n}_{S_3}.
\end{split}
\end{equation}
For $S_1$, we can invoke the Kolmogorov compactness criterion \cite{Holden_Risebro} since $v^{\D}$ converges in
$L^1_{\loc}(\Pi_T)$, and conclude that $S_1 \rightarrow 0$.  By Assumption \descref{H2}{H-2} and \descref{H4}{H-4}, ($a\in\BV(\re)$ and $u_M \in \BV(\R)$), we also have
$S_2,S_3 \rightarrow 0$.
As a result, in order to prove \eqref{Int_1} it suffices to show that
\begin{equation}\label{Int_1A}
\mathcal{I}_1:=\D x \D t \sum_{n=0}^N \sumj A(v_j^n,x_j) (\phi_{j+1}^{n} -\phi_j^n)/\D x \rightarrow
\int_0^T \int_{\R} A(v,x) \phi_x \, dx \, dt.
\end{equation}
This limit then follows from the estimate
\begin{equation}
\begin{split}
\abs{A(v_j^n,x_j)- A(v,x)} 
&\le \abs{A(v_j^n,x_j)-A(v_j^n,x)} + \abs{A(v_j^n,x) - A(v,x)}\\
&\le R(v_j^n)\abs{a(x_j)-a(x)} + L\abs{v_j^n - v},
\end{split}
\end{equation}
along with the fact that $a^{\D}\rightarrow a$ in $L^1_{\loc}(\R)$ (Lemma~\ref{lemma_a_converge}), 
and $v^{\D} \rightarrow v$ in $L^1_{\loc}(\Pi_T)$.
\end{proof}

We can now prove Theorem~\ref{theorem1}.
\begin{proof}
Taken together, Lemma~\ref{lemma_convergence} and Lemma~\ref{lemma_lim_kruzkov} establish that
the approximations $u^{\D}$ converge in $L^1_{\loc}(\Pi_T)$ and pointwise a.e. in $\Pi_T$, along a subsequence, 
to a function $u \in L^{\infty}(\Pi_T) \cap C([0,T]:L^1_{\loc}(\R))$, and $u$
is an adapted entropy solution in the sense of Definition~\ref{entropy_solution}. By Theorem~\ref{thm:uniqueness}, 
u is the unique solution to the Cauchy problem (\ref{11})--(\ref{ini12}) with initial data $u_0$. Moreover, as a consequence
of the uniqueness result, the entire computed sequence $u^{\D}$ converges to  $u$, not just a subsequence.
The final step of the proof is to extend the result to the case of $u_0 \in L^{\infty}(\R)$, as described in Section~\ref{section_compactness_direct}.
\end{proof}

\section{Appendix}
Let $g:\re\rr\re$ be defined as $g(x)=\abs{x}$.
Suppose $u$ satisfies entropy condition \eqref{AEC} then set $v(x,t)=\Psi(u(x,t),x)$ where $\Psi$ is as in \eqref{Psi}. We denote inverse of the map $\zeta\mapsto \Psi(\zeta,x)$ by $\al(\cdot,x)$. Then we have
\begin{equation}\label{inequality:modulated}
\int\limits_{0}^{\f}\int\limits_{\re}\left[\abs{\al(v(x,t),x)-\al(k,x)}\frac{\pa \phi}{\pa t}+sgn(v-k)(g(v)-g(k))\frac{\pa \phi}{\pa x}\right]\,dxdt\geq0
\end{equation}
for any $k\in\re$ and $\phi\in C_0^{\f}(\re\times\re_+)$.

\begin{lemma}[\cite{Panov2009a}]\label{lemma:v_1v_2}
	Let $v_1,v_2\in L^{\f}(\re)$ be two functions satisfying \eqref{inequality:modulated}. Then we have
	\begin{equation}
	\frac{\pa}{\pa t}\abs{v_1-v_2}+\frac{\pa}{\pa x}sgn(v_1-v_2)(g(v_1)-g(v_2))\leq0\mbox{ in }\mathcal{D}^{\p}(\re\times\re_+).
	\end{equation}
\end{lemma}
\begin{proof}
	For $0\leq \phi\in C_c^{\f}(\re_+\times\re)$ and $0\leq\psi\in C_c^{\f}(\re_+\times\re)$ we have 
	\begin{equation}\label{inequality:modulated1}
	\int\limits_{0}^{\f}\int\limits_{\re}\left[\abs{\al(v_1(x,t),x)-\al(k,x)}\frac{\pa \phi}{\pa t}+sgn(v_1-k)(g(v_1)-g(k))\frac{\pa \phi}{\pa x}\right]\,dxdt\geq0
	\end{equation}
	and
	\begin{equation}\label{inequality:modulated2}
	\int\limits_{0}^{\f}\int\limits_{\re}\left[\abs{\al(v_2(y,s),y)-\al(l,y)}\frac{\pa \psi}{\pa s}+sgn(v_2-l)(g(v_2)-g(l))\frac{\pa \psi}{\pa y}\right]\,dyds\geq0.
	\end{equation}
	Fix a $\Phi\in C_c^{\f}(\re\times\re_+)$. Let $\eta_\e$ be Friedrichs mollifiers. Consider 
	\begin{eqnarray}
	\phi(x,t)&=&\Phi\left(x,t\right)\eta_\e\left({y-x}\right)\eta_\de\left({s-t}\right),\\
	\psi(y,s)&=&\Phi\left(x,t\right)\eta_\e\left({y-x}\right)\eta_\de\left({s-t}\right).
	\end{eqnarray}
	Putting $k=v_2(y,s)$ and $l=v_1(x,t)$ in \eqref{inequality:modulated1} and \eqref{inequality:modulated2} respectively and adding the resultants we get
	\begin{equation}
	\begin{array}{lll}
	&&\int\limits_{\re\times\re_+}\int\limits_{\re\times\re_+}P_1(x,t,y,s)\frac{\pa}{\pa t}(\Phi(x,t)\eta_\e(y-x)\eta_\de(s-t))\,dtdxdsdy\\
	&+&\int\limits_{\re\times\re_+}\int\limits_{\re\times\re_+}P_2(x,t,y,s)\frac{\pa}{\pa s}(\Phi(x,t)\eta_\e(y-x)\eta_\de(s-t))\,dtdxdsdy\\
	&+&\int\limits_{\re\times\re_+}\int\limits_{\re\times\re_+}Q(x,t,y,s)\eta_\e(y-x)\eta_\de(s-t)\frac{\pa}{\pa x}\Phi(x,t),dtdxdsdy\geq0
	\end{array}
	\end{equation}
	where
	\begin{eqnarray*}
		P_1(x,t,y,s)&:=&\abs{\al(v_1(x,t),x)-\al(v_2(y,s),x)},\\
		P_2(x,t,y,s)&:=&\abs{\al(v_1(x,t),y)-\al(v_2(y,s),y)},\\
		Q(x,t,y,s)&:=&sgn(v_1(x,t)-v_2(y,s))(g(v_1(x,t))-g(v_2(y,s))).
	\end{eqnarray*}
	Let $E_0,E_1,E_2\subset\re$ be three sets such that
	\begin{eqnarray}
	E_0&:=&\left\{t\in\re_+;\mbox{ $t$ is a Lebesgue point of } v_2(x,t)\mbox{ for a.e. }x\in\re\right\},\\
	E_1&:=&\left\{x\in\re;\mbox{ $x$ is a Lebesgue point of } v_2(x,t)\mbox{ for a.e. }t\in\re_+\right\},\\
	E_2&:=&\left\{x;\lim\limits_{\e\rr0}\int\eta_\e\left({x-y}\right)\max\limits_{\abs{u}\leq r}\abs{\al(u,x)-\al(u,y)}=0\right\}.
	\end{eqnarray}
	Since $v_2\in L^{\f}(\re\times\re_+)$, $E_0,E_1$ are measurable sets and $meas(\re\setminus E_0)=meas(\re\setminus E_1)=0$. By our assumption, for a fixed $x\in\re$, $\Psi(x,\cdot)$ is Lipschitz on $[-r,r]$. Since $C([-r,r])$ is separable, by Pettis Theorem we have measurability of $E_2$ and $meas(\re\setminus E_2)=0$. Therefore we can get
	\begin{eqnarray}
	&&\abs{\int_{\re}P_1(x,t,y,s)\eta_{\e}(y-x)\,dy-P_1(x,t,x,s)}\\
	&&\leq\int\limits_{\re}\abs{\al(v_2(y,s),x)-\al(v_2(x,s),x)}\eta_\e(y-x)\,dy\rr0\label{P1}
	\end{eqnarray}
	as $\e\rr0$ for $x\in E_1$ and a.e. $t,s\in\re_+$. We can also obtain
	\begin{eqnarray}
	&&\abs{\int\limits_{\re}P_2(x,t,y,s)\eta_\e(y-x)dy-P_2(x,t,x,s)}\nonumber\\
	&\leq&\int\limits_{\re}\abs{\al(v_2(y,s),x)-\al(v_2(x,s),x)}\eta_\e(y-x)\,dy+2\int\limits_\re\eta_\e\left({x-y}\right)\max\limits_{\abs{u}\leq r}\abs{\al(u,x)-\al(u,y)}\nonumber\\
	&\rr&0\label{P2}
	\end{eqnarray}
	as $\e\rr0$ for $x\in E_2$ and a.e. $t,s\in\re_+$. With the help of \eqref{P1} and \eqref{P2} and Lebesgue Dominated Convergence Theorem we have
	\begin{eqnarray}
	\lim\limits_{\e\rr0}\int\limits_{\re^2_+}\int\limits_{\re^2_+}\left(P_1(x,t,y,s)\frac{\pa}{\pa t}(\Phi(x,t)\eta_{\de}(s-t))+
	P_2(x,t,y,s)\frac{\pa}{\pa s}(\Phi(x,t)\eta_{\de}(s-t))\right)\\
	\eta_\e(y-x)\,dtdxdsdy
	=\int\limits_{\re^2_+}\int\limits_{\re_+}P_1(x,t,x,s)\eta_{\de}(s-t)\frac{\pa}{\pa t}\Phi(x,t)\,dtdxds.
	\end{eqnarray}
	In a similar way we can show 
	\begin{equation}
	\lim\limits_{\de\rr0}\int\limits_{\re^2_+}\int\limits_{\re_+}P_1(x,t,x,s)\eta_{\de}(s-t)\frac{\pa}{\pa t}\Phi(x,t)\,dtdxds=\int\limits_{\re^2_+}P_1(x,t,x,t)\frac{\pa}{\pa t}\Phi(x,t)\,dtdx.
	\end{equation}
	Similarly we have 
	\begin{equation}
	\abs{\int\limits_{\re}Q(x,t,y,s)\eta_\e(y-x)\,dy-Q(x,t,x,s)}\leq2C_R\int\limits_{\re}\abs{v_2(y,s)-v_2(x,s)}\eta_{\e}(y-x)\,dy\rr0
	\end{equation}
	as $\e\rr0$ for $x\in E_1$ and a.e. $t,s\in\re_+$. Then by Lebesgue Dominated Convergence Theorem we have
	\begin{equation}
	\lim\limits_{\e\rr0}\int\limits_{\re^2_+}\int\limits_{\re^2_+}Q(x,t,y,s)\eta_\e(y-x)\eta_\de(s-t)\frac{\pa }{\pa x}\Phi
	\,dtdxdsdy=\int\limits_{\re^2_+}\int\limits_{\re_+}Q(x,t,x,s)\eta_\de(s-t)\frac{\pa }{\pa x}\Phi\,dtdxds.
	\end{equation}
	We also have for a.e. $x\in\re$ and $t\in E_0$
	\begin{equation}
	\abs{\int\limits_{\re}Q(x,t,x,s)\eta_\de(y-x)\,dy-Q(x,t,x,t)}\leq2C_R\int\limits_{\re}\abs{v_2(x,t)-v_2(x,s)}\eta_{\de}(s-t)\,ds\rr0
	\end{equation}
	as $\de\rr0$.
	This yields
	\begin{equation}
	\lim\limits_{\e\rr0,\de\rr0}\int\limits_{\re^2_+}\int\limits_{\re^2_+}Q(x,t,y,s)\eta_\e(y-x)\eta_\de(s-t)\frac{\pa}{\pa x}\Phi\,dtdxdsdy=\int\limits_{\re^2_+}Q(x,t,x,t)\,dtdx.
	\end{equation}
	This completes the proof.
\end{proof}
Observe the following 
\begin{equation}
g(v(x,t))=g(\Psi(u(x,t),x))=A(u(x,t),x)=A(\al(v(x,t),x),x).
\end{equation}
From Lemma \ref{lemma:v_1v_2} we can prove the following by a similar argument as in \cite{Kruzkov}.
\begin{lemma}
	Let $v_1,v_2\in C([0,T],L^{1}_{loc}(\re))\cap L^{\f}(\re\times\re_+)$ be two function satisfying \eqref{inequality:modulated}. Then for a.e. $t\in[0,T]$ and any $r>0$ we have
	\begin{equation}
	\int\limits_{\abs{x}\leq r}\abs{\al(v_1(x,t),x)-\al(v_2(x,t),x)}\,dx\leq \int\limits_{\abs{x}\leq r+L_1t}\abs{\al(v_1(x,0),x)-\al(v_2(x,0),x)}\,dx
	\end{equation}
	where $L_1:=\sup\{\pa_u A(u,x);\,x\in\re,\abs{u}\leq\max(\|v_1(x,0)\|_{L^{\f}},\|v_2(x,0)\|_{L^{\f}})\}$.
\end{lemma}

\noindent\textbf{Acknowledgement.}  First author would like to thank Inspire faculty-research grant\\ DST/INSPIRE/04/2016/000237.

\end{document}